\documentclass[a4paper,twoside,11pt]{amsart}
\RequirePackage[english]{babel}
\usepackage[latin1]{inputenc}
\usepackage{amsmath}
\usepackage{amssymb}
\usepackage{amsfonts}
\usepackage{mathrsfs}
\usepackage{fixltx2e}[2005/12/01]
\usepackage{xcolor}
\usepackage[colorlinks=true]{hyperref}

\usepackage{enumerate}
\renewcommand{\theenumi}{{\upshape{(\roman{enumi})}}}

\usepackage[matrix,arrow,cmtip]{xy} 
\SelectTips{cm}{}
\newdir{(}{{}*!/-5pt/@^{(}}
\newdir{(x}{{}*!/-5pt/@_{(}}
\newdir{+}{{}*!/-9pt/{}}
\newdir{>+}{@{>}*!/-9pt/{}}
\entrymodifiers={+!!<0pt,\fontdimen22\textfont2>}

\usepackage[nosubsections]{mytheorems2}
\theoremstyle{ddef}
\newtheorem*{definition*}{Definition}

\newtheoremstyle{dtheoremnopar}{3 mm}{1 mm}{\itshape}{}{\bfseries}{.}{ }
{\thmname{#1}\thmnumber{ #2}\thmnote{ \mdseries(#3)\bfseries}}

\theoremstyle{dtheoremnopar}
\newcounter{theoremx}
\newtheorem{theoremalpha}[theoremx]{Theorem}
\newtheorem{corollaryalpha}[theoremx]{Corollary}

\newcommand{\spref}[1]{\href{http://stacks.math.columbia.edu/tag/#1}{#1}}
\newcommand{\spcite}[1]{\cite[\spref{#1}]{stacks-project}}
\newcommand\ZZ{\mathbb{Z}}
\newcommand\FF{\mathbb{F}}
\newcommand\inj{\hookrightarrow}
\newcommand\surj{\twoheadrightarrow}
\newcommand\map[3]{#1\colon #2\rightarrow #3}
\DeclareMathOperator{\Aut}{Aut}
\DeclareMathOperator{\Sym}{Sym}
\DeclareMathOperator{\coker}{coker}
\newcommand\sF{\mathcal{F}}
\newcommand\sG{\mathcal{G}}
\newcommand\sI{\mathcal{I}}
\newcommand\sJ{\mathcal{J}}
\newcommand\sK{\mathcal{K}}
\newcommand\sN{\mathcal{N}}
\newcommand\sO{\mathcal{O}}
\newcommand\sP{\mathcal{P}}
\renewcommand\AA{\mathbb{A}}    
\DeclareMathOperator{\Spec}{Spec}
\DeclareMathOperator{\Supp}{Supp}  
\newcommand{\Gm}{\mathbb{G}_m}
\newcommand{\Gmu}{\pmb{\mu}} 
\newcommand{\gitq}{/\!\!/} 
\newcommand\QCoh{\mathbf{QCoh}} 

\newcommand{\stX}{\mathscr{X}}
\newcommand{\stY}{\mathscr{Y}}
\newcommand{\stZ}{\mathscr{Z}}
\newcommand{\stW}{\mathscr{W}}

\newcommand{\stG}{\mathscr{G}} 
\newcommand{\stC}{\mathscr{C}}
\newcommand{\stM}{\mathscr{M}}
\newcommand{\stE}{\mathscr{E}}
\newcommand{\coho}{\mathcal{H}}
\newcommand{\LDERF}{\mathbf{L}}
\newcommand{\RDERF}{\mathbf{R}}
\newcommand{\Lotimes}{\overset{\mathbf{L}}{\otimes}}
\newcommand{\LL}{\mathbb{L}} 
\newcommand{\Dqc}{\mathbf{D}_{\mathrm{qc}}} 
\newcommand{\Dbcoh}{\mathbf{D}^-_{\mathrm{Coh}}} 

\newcommand{\itemref}[1]{\ref{#1}}
\DeclareMathOperator{\stab}{stab}

\begin{document}

\title[Luna's fundamental lemma for stacks]{A generalization of Luna's fundamental lemma for stacks with good moduli spaces}
\author{David Rydh}
\address{KTH Royal Institute of Technology\\Department of 
  Mathematics\\SE-100 44 Stockholm\\Sweden}
\email{dary@math.kth.se}
\date{2020-08-25}
\thanks{Supported by the Swedish Research Council grant no 2011-5599 and 2015-05554.}

\begin{abstract}
We generalize Luna's fundamental lemma to smooth morphisms between stacks with
good moduli spaces. We also give a precise condition for when it holds for
non-smooth morphisms and versions for coherent sheaves and complexes.
This generalizes earlier results by Alper,
Abramovich--Temkin, Edidin and Nevins.
\end{abstract}

\maketitle


\setcounter{secnumdepth}{0}
\begin{section}{Introduction}
Let $G$ be a linearly reductive group acting on affine varieties $X=\Spec B$
and $Y=\Spec A$. Then $X$ and $Y$ admit \emph{good quotients} $\pi_X\colon X\to
X\gitq G=\Spec B^G$ and $\pi_Y\colon Y\to Y\gitq G=\Spec A^G$. These are not
orbit spaces in general but the closed points of the quotients correspond to
closed orbits.  Let $f\colon X\to Y$ be a $G$-equivariant morphism and let
$f\gitq G\colon X\gitq G\to Y\gitq G$ be the induced morphism.  When $f$ is
\'etale, Luna's fundamental lemma~\cite[p.~94]{luna} gives a criterion for $f$
to be \emph{strongly \'etale}, that is, $f\gitq G$ is \'etale and $f$ is the
pull-back of $f\gitq G$ along $\pi_Y$. An analogous criterion when $f$ is not
\'etale was recently given by Abramovich and
Temkin~\cite{abramovich-temkin_luna-fundamental} when $G$ is diagonalizable.

Stacks with \emph{good moduli spaces} generalize good quotients: a good
quotient $X\to X\gitq G$ gives rise to a good moduli space $[X/G]\to X\gitq G$.
Luna's fundamental lemma for stacks with good moduli spaces has been
generalized for \'etale morphisms \cite[Thm.~6.10]{alper_loc-quot-struct} (also
see \cite[Thm.~3.14]{alper-hall-rydh_etale-local-stacks}) and for closed
regular immersions \cite{edidin_strong-regular-embeddings}.

The purpose of this article is to give a general formulation of Luna's
fundamental lemma for stacks with good moduli spaces that simultaneously
generalize all the results mentioned above.

\subsection*{Good moduli spaces}
We briefly recall some properties of good moduli spaces (see
Section~\ref{S:gms}). Let $\stX$ be an algebraic stack that admits a good moduli
space $\pi_X\colon \stX\to X$. In particular, $X$ is an algebraic space and
$\pi_X$ is initial among maps to algebraic spaces. The map $\pi_X$ is universally
closed and for every
point $x\in |X|$, there is a unique closed point $x_0\in \pi_X^{-1}(x)$. We say
that such a point is \emph{special}.

Let $f\colon \stX\to \stY$ be a morphism between algebraic stacks
with good moduli spaces $X$ and $Y$. This induces a commutative diagram
\begin{equation*}
\vcenter{%
\xymatrix{%
\stX\ar[r]^f\ar[d] & \stY\ar[d] \\
X\ar[r]^g & Y.}
}%
\end{equation*}

We say that $f$ is \emph{strong} if the diagram above is cartesian.
We say that $f$ is \emph{special}, if $f$ takes special points to special points.

\begin{definition*}
Let $f\colon \stX\to \stY$ be a morphism of algebraic stacks and consider the
induced morphism $\varphi\colon I_{\stX}\to f^*I_{\stY}$ of inertia stacks. We
say that $f$ is
\begin{enumerate}
\item \emph{stabilizer preserving} if $\varphi$ is an isomorphism;
\item \emph{fiberwise stabilizer preserving at $y\in |\stY|$} if $\varphi|_{f^{-1}(\stG_y)}$ is an isomorphism; and
\item \emph{pointwise stabilizer preserving at $x\in |\stX|$} if $\varphi|_{\stG_x}$ is an isomorphism.
\end{enumerate}
\end{definition*}

A strong morphism is stabilizer preserving and special.
We can now state the main theorem.

\begin{theoremalpha}\label{T:MAIN-THEOREM}
Let $f\colon \stX\to \stY$ be a morphism between algebraic stacks that
admit good moduli spaces $\pi_X\colon \stX\to X$ and $\pi_Y\colon \stY\to Y$.
Let $g\colon X\to Y$ denote the induced morphism. Assume that $\pi_X$
and $\pi_Y$ have affine diagonals. Further assume either that $f$ is locally of
finite presentation or that $\stX$ and $\stY$ are locally noetherian.
Then $f$ is strong if and only if
\begin{enumerate}
\item\label{TI:special}
  $f$ is special,
\item\label{TI:fw-stab-pres}
  $f$ is fiberwise stabilizer preserving at every special point of $\stY$, and
\item\label{TI:1-L-strong}
  for every special point $x_0\in |\stX|$, the (ind-)vector bundle
  $\coho^{-1}(\LDERF i_{x_0}^*\LL_f)$ is trivial where $i_{x_0}\colon \stG_{x_0}\inj \stX$ denotes
  the inclusion of the residual gerbe.
\end{enumerate}
If in addition $f$ has one of the properties:
\begin{enumerate}\renewcommand{\theenumi}{{\upshape{(\alph{enumi})}}}
\item regular, smooth, \'etale, open immersion,\label{TI:first-reduced}\label{TI:regular}
\item unramified, closed immersion, locally closed immersion,
  quasi-regular immersion, Koszul-regular immersion,
  monomorphism,\label{TI:last-reduced}
\item flat, syntomic, local complete intersection, quasi-finite, finite,
  locally of finite presentation;
\end{enumerate}
then so has $g$. Moreover
\begin{enumerate}\renewcommand{\theenumi}{{\upshape{(\arabic{enumi})}}}
\item If $\pi_Y$ is a coarse moduli space, then condition \ref{TI:special}
  is redundant.
\item If $f$ has reduced fibers, e.g., if $f$ has one of the
properties in~\ref{TI:first-reduced}--\ref{TI:last-reduced}, then we can replace
condition~\ref{TI:fw-stab-pres} with the condition that $f$ is
\emph{pointwise} stabilizer
preserving at every point of $\stX$ above a special point of $\stY$.
\item If $f$ is flat, then condition~\ref{TI:1-L-strong} is redundant.
\end{enumerate}
\end{theoremalpha}

This generalizes the main theorem of \cite{abramovich-temkin_luna-fundamental}
in several different directions. Most importantly, we do not require that
$\stX$ and $\stY$ are quotient stacks, nor that the stabilizer groups are
diagonalizable. This answers
\cite[\S1.5]{abramovich-temkin_luna-fundamental}. Furthermore, we prove that
flatness descends and that the condition on cotangent complexes is redundant
for flat morphisms. We also give pointwise variants of conditions
\ref{TI:special}--\ref{TI:1-L-strong}
and show that if these hold at a special point $x_0\in \pi_X^{-1}(x)$, then $f$
is strong in a neighborhood of $x\in |X|$ (Theorem~\ref{T:main-thm:local}).
In particular, if $X$ is quasi-compact and quasi-separated, then it is enough
to verify \ref{TI:1-L-strong} for closed points $x_0$.
Finally, we
also allow $\stX$ and $\stY$ to be non-noetherian if we instead assume that $f$
is locally of finite presentation.

Recall that syntomic means flat with fibers that are lci
(local complete intersections).
We do not assume that lci maps and syntomic maps between noetherian algebraic
stacks
are of finite type, see Section~\ref{S:lci}. In the noetherian setting,
the notions of regular, Koszul-regular and quasi-regular coincide
(see Section~\ref{S:reg-emb}). A Koszul-regular immersion is
an lci immersion. 

We also have a similar result for coherent sheaves. This is well-known for
vector bundles, see~\cite[Thm.~10.3]{alper_good-mod-spaces}
and~\cite[\S2]{drezet-narasimhan}. For coherent sheaves this was proved in
the GIT setting by Nevins~\cite[Thm.~1.2]{nevins_descent}.

\begin{theoremalpha}\label{T:STRONG-MODULES}
Let $\pi\colon \stX\to X$ be a good moduli space. Let $\sF$ be a
quasi-coherent $\sO_{\stX}$-module of finite presentation. The following
are equivalent
\begin{enumerate}
\item\label{CI:STRONG-MODULES:pullback}
  $\sF=\pi^*\sG$ for some $\sO_X$-module $\sG$ of finite presentation;
\item\label{CI:STRONG-MODULES:counit-bijective}
  the counit map $\epsilon\colon \pi^*\pi_*\sF\to \sF$ is bijective; and
\item\label{CI:STRONG-MODULES:trivial-actions}
  for every special point $x_0\in |\stX|$, the vector bundles
  $i_{x_0}^*\sF$ and $\coho^{-1}(\LDERF i_{x_0}^*\sF)$ are trivial where
  $i_{x_0}\colon \stG_{x_0}\inj \stX$ denotes the inclusion of the residual gerbe.
\end{enumerate}
Under these equivalent conditions, $\sG=\pi_*\sF$ is of finite presentation.
Moreover, if $\sF$ is flat, i.e., locally free, then it is enough that $i_{x_0}^*\sF$
is trivial for every special point $x_0$ and then $\pi_*\sF$ is locally
free as well.
\end{theoremalpha}

See Theorem~\ref{T:strong-modules:gen} for a more general version
of Theorem~\ref{T:STRONG-MODULES}. Also see
Theorem~\ref{T:strong-modules:complexes} for a version for complexes
generalizing~\cite[Thm.~1.3]{nevins_descent}.

\subsection{Applications to tame stacks}
Recall that a tame stacks is a stack with finite inertia and finite linearly
reductive stabilizers~\cite{abramovich_olsson_vistoli_tame_stacks}.
The coarse moduli space of a tame stack is also a good
moduli space.

\begin{corollaryalpha}\label{C:MAPS-TO-STACKS}
Let $\stX$ be a tame stack with coarse moduli space $\pi\colon \stX\to X$.
Let $\stY$ be any
algebraic stack and let $f\colon \stX\to \stY$ be a morphism. The
following are equivalent
\begin{enumerate}
\item $f$ factors through $\pi$,
\item $f$ factors through the fppf sheafification $\stX\to \pi_0\stX$,
\item the induced morphism $I_f\colon I_{\stX}\to I_{\stY}$ of inertia stacks is
  trivial, i.e., 
  factors through the identity section $\stY\to I_{\stY}$.
\item\label{CI:MAPS:trivial-actions}
  for every field $k$, every morphism $x\colon \Spec k\to \stX$ and every
  automorphism $\varphi\in\Aut(x)$, the automorphism $f(\varphi)\in\Aut(f\circ
  x)$ is the identity.
\end{enumerate}
Under these equivalent conditions, the factorization through $\pi$ is unique.
\end{corollaryalpha}


Concretely, if for example $\stY=\stM_g$ is the stack of smooth genus $g$
curves, then a family of genus $g$ curves $\stC\to \stX$ comes from family of
genus $g$ curves $C\to X$ if, for every point $x\in \stX(k)$, the stabilizer group
of $x$ acts trivially on $\stC_x$. Taking $\stY=BGL_n$ we recover Theorem~\ref{T:STRONG-MODULES} for vector bundles on tame stacks.  Taking
$\stY=BG$ for a flat group scheme $G\to X$ locally of finite presentation,
we obtain:

\begin{corollaryalpha}\label{C:TORSORS}
Let $\stX$ be a tame stack with coarse moduli space $\pi\colon \stX\to X$.
Let $G\to X$ be a flat group scheme locally of finite presentation.
Then $\pi^*$ induces an equivalence of
categories between the category of $G$-torsors $E\to X$ and the category of
$G$-torsors $\stE\to \stX$ such that $\Aut(x)$ acts trivially on
$\stE_x$ for every point $x\in \stX(k)$.
\end{corollaryalpha}

The special cases $G=\Gm$, $G=S_n$ and $G$ tame are \cite[Props.~6.1, 6.2 and
  6.4]{olsson_G-torsors}. When $G$ is defined over a field, the result is
\cite[Prop.~3.15]{biswas-borne_tamely-ramified-torsors-and-parabolic-bundles}.
For $G$ \'etale, the result is also true for non-tame coarse
moduli spaces. This follows directly from
\cite[Prop.~6.7]{rydh_finite_quotients}.
If $X$ is quasi-compact and quasi-separated, then it is enough to verify
conditions Theorem~\ref{T:STRONG-MODULES}\ref{CI:STRONG-MODULES:trivial-actions}, Corollary~\ref{C:MAPS-TO-STACKS}\ref{CI:MAPS:trivial-actions} and Corollary~\ref{C:TORSORS} at closed points.

\subsection{Necessity of conditions}
When $f$ is \'etale, Luna's fundamental lemma is also true for coarse moduli
spaces \cite[Prop.~6.7]{rydh_finite_quotients} and more generally for adequate
moduli spaces \cite[Thm.~3.14]{alper-hall-rydh_etale-local-stacks}.
Adequate moduli spaces gives a full generalization of GIT in positive
characteristic, also allowing for reductive groups that are not linearly
reductive.

When $f$ is not \'etale, these generalizations fail: in
\cite[\S4.5]{romagny-rydh-zalamansky_complexity} there is an example of an
action of $G=\ZZ/p\ZZ$ on an affine scheme $\Spec A$, where
$A=\FF_p[\epsilon,x]/(\epsilon^2)$, such that
\begin{enumerate}
\item Theorem~\ref{T:MAIN-THEOREM} does not hold for $\stY=[\Spec A/G]$ and a
  certain affine smooth morphism $f$,
\item Theorem~\ref{T:STRONG-MODULES} does not hold for $\stX=[\Spec A/G]$ and a
  certain torsion line bundle $\sF$,
\item Corollary~\ref{C:MAPS-TO-STACKS} does not hold for $\stX=[\Spec A/G]$ and
  $\stY=B\Gmu_p$ or $\stY=B\Gm$.
\item Corollary~\ref{C:TORSORS} does not hold for $\stX=[\Spec A/G]$ and
  $G=\Gmu_p$.
\end{enumerate}

If $\pi_Y$ does not have separated diagonal, then Theorem~\ref{T:MAIN-THEOREM}
fails. Consider the non-separated group scheme $G\to \AA^1$ which as a scheme
is the affine line with a double origin. We have a surjective homomorphism
$(\ZZ/2\ZZ)_{\AA^1}\to G$. Let $\Gm$ act by multiplication on $\AA^1$ and
let $\ZZ/2\ZZ$ and $G$ act trivially on $\AA^1$. Then
$f\colon [\AA^1 / (\ZZ/2\ZZ \times \Gm)] \to [\AA^1 / (G \times \Gm)]$ satisfies
(i)--(iii) of Theorem~\ref{T:MAIN-THEOREM} but is not stabilizer preserving.
In \cite[Thm.~13.1]{alper-hall-rydh_etale-local-stacks}, it is shown under some
mild finiteness assumptions that if $\pi_Y$ has separated diagonal, then it has
affine diagonal.

\subsection*{Overview}
In Section~\ref{S:gms} we recall
the notion of a good moduli space and its basic properties, including a
variant of Nakayama's lemma (Lemma~\ref{L:Nakayama-gms}).
We also prove Theorem~\ref{T:STRONG-MODULES} and Theorem~\ref{T:MAIN-THEOREM}
when $f$ is a closed immersion.
In Section~\ref{S:special}, we prove
that conditions (i) and (ii) in Theorem~\ref{T:MAIN-THEOREM} imply that
$\rho\colon \stX\to
\stY\times_Y X$ is a closed immersion. This is the key to reduce
Theorem~\ref{T:MAIN-THEOREM} to the special case treated in
Section~\ref{S:gms}.

In Section~\ref{S:cotangent-complexes} we study condition (iii) and
in Section~\ref{S:pf-of-main-thm} we give the local version of
the main theorem (Theorem~\ref{T:main-thm:local}) and most of the properties.

In Sections~\ref{S:reg-emb}--\ref{S:lci} we briefly recall the notions of
regular immersions and local complete intersections with an emphasis on
stacks, maps between non-noetherian schemes/stacks, and maps between noetherian
schemes/stacks that are not of finite type. We then prove
that the properties quasi-regular, Koszul-regular and lci descend
for strong morphisms.

\subsection*{Acknowledgments}
This paper was greatly influenced by reading Dan Edidin's
paper~\cite{edidin_strong-regular-embeddings} on strong regular immersions and
the paper by Dan Abramovich and Michael
Temkin~\cite{abramovich-temkin_luna-fundamental} mentioned earlier. Likewise,
Corollary~\ref{C:MAPS-TO-STACKS} was inspired by Matthieu Romagny and
Gabriel Zalamansky~\cite{romagny-rydh-zalamansky_complexity} and
Corollary~\ref{C:TORSORS} by Indranil Biswas and Niels
Borne~\cite{biswas-borne_tamely-ramified-torsors-and-parabolic-bundles}. I
would also like to thank Jarod Alper, Daniel Bergh and Jack Hall
for useful discussions.
\end{section}
\setcounter{secnumdepth}{3}


\begin{section}{Good moduli spaces}\label{S:gms}
In this section we briefly review the theory of good moduli spaces, due
to Alper~\cite[\S4]{alper_good-mod-spaces} and prove
Theorem~\ref{T:STRONG-MODULES}.

\subsection{Good moduli spaces}
Let $\stX$ be an algebraic stack. We say that $\stX$ admits a \emph{good moduli
space} if there exists an algebraic space $X$
and a morphism $\pi_X\colon \stX\to X$ such that
\begin{enumerate}
\item $\pi_X$ is quasi-compact and quasi-separated,
\item $\pi_X$ is cohomologically affine,
i.e., the functor $(\pi_X)_*\colon \QCoh(\stX)\to\QCoh(X)$ is
exact, and
\item $(\pi_X)_*\sO_\stX=\sO_X$.
\end{enumerate}
\subsubsection*{Uniqueness}
The space $X$, if it exists, is
unique~\cite[Thm.~6.6]{alper_good-mod-spaces}
(see~\cite[Thm.~3.12]{alper-hall-rydh_etale-local-stacks} for the non-noetherian case).

\subsubsection*{Noetherian and coherence}
If $\stX$ is noetherian, then so is
$X$ and $(\pi_X)_*$ preserves coherence~\cite[Thm.~4.16 (x)]{alper_good-mod-spaces}.
If in addition $\pi_X$ has affine diagonal (e.g., $\stX$ has affine diagonal),
then $\pi_X$ is of finite type~\cite[Thm.~A.1]{alper-hall-rydh_luna-stacks}.

\subsubsection*{Quasi-coherent sheaves}
The functor $\pi_X^*$ is fully faithful, that is, the unit map $\sF\to
(\pi_X)_*(\pi_X)^*\sF$ is an isomorphism for every quasi-coherent $\sO_X$-module
$\sF$ \cite[Prop.~4.5]{alper_good-mod-spaces}. This follows from the
definitions after choosing a free presentation of $\sF$ locally on $X$.

\subsubsection*{Base change}
If we have a cartesian diagram
\[
\xymatrix{%
\stX\ar[r]^f\ar[d]_{\pi_X} & \stY\ar[d]^{\pi_Y} \\
X\ar[r]^g & Y\ar@{}[ul]|\square}
\]
where $X$ and $Y$ are algebraic spaces and $\pi_Y$ is a good moduli space,
then $\pi_X$ is a good moduli space and push-forward along $\pi_Y$ commutes with
base change: $g^*(\pi_Y)_*=(\pi_X)_*f^*$~\cite[Prop.~4.7]{alper_good-mod-spaces}.
If $f$ is affine, then so is $g$ and $g_*\sO_X=(\pi_Y)_*f_*\sO_\stX$.

\subsubsection*{Purity}
The good moduli space morphism $\pi_X\colon \stX\to X$ is \emph{pure}, that is,
if $\pi_{X'}\colon \stX'\to X'$ is the base change of $\pi_X$ along any map
$X'\to X$, then $\sO_{X'}\to (\pi_{X'})_*\sO_{\stX'}$ is injective. This follows
from the definition of a good moduli space since $\pi_{X'}$ also is a good
moduli space.

\subsubsection*{Residual gerbes}
Let $x\in |\stX|$ be a point. If $\stX$ is quasi-separated (or equivalently,
$X$ is quasi-separated), then the \emph{residual gerbe}
$\stG_x$ exists and comes along with a quasi-affine monomorphism
$i_x\colon \stG_x\inj \stX$ \cite[Thm.~B.2]{rydh_etale-devissage}.
The residual gerbe $\stG_x$ is an fppf gerbe over the \emph{residue field}
$\kappa(x)$.
When $x$ is a closed point, $i_x$ is a
closed immersion corresponding to the maximal ideal of $x$.

\subsubsection*{Topological properties}
The morphism $\pi_X$ is universally closed.
For every point $x\in |X|$, there is a unique closed point $x_0\in
\pi^{-1}(x)$~\cite[Thm.~4.16 (iv)]{alper_good-mod-spaces}. We say that such a
point is \emph{special}. If $\stZ\inj \stX$ is a closed substack, then
$Z:=\pi_X(\stZ)$ is a closed subspace of $X$ and $\stZ\to Z$ is a good moduli
space. In particular, if $x_0$ is
special, then the residual gerbe $\stG_{x_0}$ has good moduli space $\Spec
\kappa(x)$. In particular, the residue fields $\kappa(x)=\kappa(x_0)$ coincide.

\subsection{Nakayama's lemma}
Let $\stX$ be an algebraic stack and $\sF$ a quasi-coherent $\sO_\stX$-module
of \emph{finite type}. Nakayama's lemma then implies that $\Supp \sF$ is
closed and that $x\in \Supp \sF$ if and only if $i_x^*\sF=0$. In particular,
if $i_x^*\sF=0$, then $\sF|_U=0$ for an open neighborhood $U$ of $x$.
If $\varphi\colon \sF \to \sG$ is a homomorphism of
quasi-coherent $\sO_\stX$-modules and $\sG$ is of finite type, then
$\coker \varphi$ is of finite type and we conclude that $\varphi$ is surjective
in a neighborhood of $x$ if and only if $i_x^*\varphi$ is surjective.

We have the following version of Nakayama's lemma for stacks with
good moduli spaces.

\begin{lemma}\label{L:Nakayama-gms}
Let $\stX$ be an algebraic stack with good moduli space $\map{\pi}{\stX}{X}$.
Let $x\in |X|$ be a point and $x_0\in |\stX|$ be the unique special
point above $x$.
\begin{enumerate}
\item\label{LI:Nakayama-gms:zero}
  Let $\sF$ be a quasi-coherent $\sO_\stX$-module of finite type.
  If $i_{x_0}^*\sF=0$, then $\sF|_{\pi^{-1}(U)}=0$ for an open neighborhood
  $U$ of $x$.
\item\label{LI:Nakayama-gms:surj}
  Let $\varphi\colon \sF\to \sG$ be a homomorphism of quasi-coherent
  $\sO_\stX$-modules where $\sG$ is of finite type. If $i_{x_0}^*\varphi$ is
  surjective, then $\varphi|_{\pi^{-1}(U)}$ is surjective for an open
  neighborhood $U$ of $x$.
\end{enumerate}
\end{lemma}
\begin{proof}
For \ref{LI:Nakayama-gms:zero}, take
$U=X\smallsetminus \pi(\Supp \sF)$ and for \ref{LI:Nakayama-gms:surj} consider
$\coker \varphi$.
\end{proof}

\subsection{Proof of Theorem~\ref{T:STRONG-MODULES}}
We will need the following lemma.

\begin{lemma}\label{L:surjective}
Let $\stX$ be an algebraic stack with good moduli space $\pi\colon \stX\to X$.
Let $\sF$ be a quasi-coherent $\sO_\stX$-module. Let $x_0\in |\stX|$ be a special
point and $x=\pi(x_0)$.
  If $\sF$ is of finite type and $i_{x_0}^*\sF$ is trivial, then
  after base change along an \'etale neighborhood $X'\to X$ of $x$
  there exists a surjective homomorphism
  \[
  \varphi\colon \sO_{\stX}^{\oplus n}\surj \sF
  \]
  such that $i_{x_0}^*\varphi$ is an isomorphism.
\end{lemma}
\begin{proof}
First assume that $X$ is local henselian with closed point $x$. By
assumption $i_{x_0}^*\sF=\sO_{\stG_{x_0}}^{\oplus n}$. The surjective map $\sO_\stX^{\oplus
  n}\surj (i_{x_0})_*\sO_{\stG_{x_0}}^{\oplus n} = (i_{x_0})_*i_{x_0}^*\sF$ lifts to a map $\varphi
\colon \sO_\stX^{\oplus n} \to \sF$. Indeed, $\Gamma(\stX,\sF)\to
\Gamma(\stX,(i_{x_0})_*i_{x_0}^*\sF)$ is surjective since $\stX$ is cohomologically
affine. By construction $i_{x_0}^*\varphi$ is an isomorphism.
For general $X$ we then obtain, after replacing $X$ with an \'etale
neighborhood, a homomorphism $\varphi\colon \sO_\stX^{\oplus n} \to \sF$ such
$i_{x_0}^*\varphi$ is an isomorphism. Then $\varphi$ is surjective after
replacing $X$ with an open neighborhood of $x$ by
Lemma~\ref{L:Nakayama-gms}\ref{LI:Nakayama-gms:surj}.
\end{proof}

We now state and prove the following generalization of
Theorem~\ref{T:STRONG-MODULES}.

\begin{theorem}\label{T:strong-modules:gen}
Let $\stX$ be an algebraic stack with good moduli space $\pi\colon \stX\to X$.
Let $\sF$ be a quasi-coherent $\sO_\stX$-module. Let $x\in |X|$ and let
$x_0\in |\stX|$ be the unique special point above $x$.
Let $\epsilon\colon \pi^*\pi_*\sF\to \sF$ denote the counit map.
\begin{enumerate}
\item\label{TI:strong:surj}
  If $\sF$ is of finite type, then $\epsilon$ is surjective in
  a neighborhood of $x$ if and only if $i_{x_0}^*\sF$ is a trivial vector bundle.
  If this is the case, then $\pi_*\sF$ is of finite type in a neighborhood of
  $x$.
\item\label{TI:strong:iso}
  If $\sF$ is of finite presentation, then $\epsilon$ is an isomorphism in
  a neighborhood of $x$ if and only if $i_{x_0}^*\sF$ and $\coho^{-1}(\LDERF i_{x_0}^*\sF)$
  are trivial vector bundles. If this is the case, then $\pi_*\sF$ is of finite
  presentation in a neighborhood of $x$.
\item\label{TI:strong:vb}
  If $\sF$ is a vector bundle, then the following are equivalent
\begin{enumerate}
\item $\epsilon$ is an isomorphism in a neighborhood of $x$;
\item $\epsilon$ is surjective in a neighborhood of $x$;
\item $i_{x_0}^*\sF$ is a trivial vector bundle;
\end{enumerate}
and under these conditions, $\pi_*\sF$ is a vector bundle in a neighborhood
of $x$.
\end{enumerate}
\end{theorem}

\begin{proof}
The questions are local on $X$ so we may assume that $X$ is affine.  If
$\epsilon$ is surjective in a neighborhood of $x$, then $i_{x_0}^*\epsilon$ is
surjective so $i_{x_0}^*\sF$ is trivial. Indeed,
$i_{x_0}^*\pi^*\pi_*\sF=i_x^*\pi_*\sF\otimes_{\kappa(x)}\sO_{\stG_{x_0}}$ is always
trivial.  Conversely, if $\sF$ is of finite type and $i_{x_0}^*\sF$ is trivial,
then by Lemma~\ref{L:surjective} there is, after replacing $X$ with a neighborhood, a
surjective homomorphism $\varphi\colon \sO_\stX^{\oplus n}\to \sF$.  It follows
that $\epsilon$ is surjective.  Since $\pi_*$ is exact we also get a surjective
homomorphism $\pi_*\varphi\colon \sO_X^{\oplus n}\to \pi_*\sF$ so $\pi_*\sF$ is
of finite type.

If $\sF$ is of finite presentation and the conditions in \ref{TI:strong:surj}
are satisfied,
then we can find a surjective homomorphism $\varphi\colon \sO_\stX^{\oplus
  n}\to \sF$ such that $i_{x_0}^*\varphi$ is an isomorphism.  Let
$\sK=\ker(\varphi)$ which is of finite type. We note that $\coho^{-1}(\LDERF
i_{x_0}^*\sF)=i_{x_0}^*\sK$ and that $\epsilon_\sK\colon \pi^*\pi_*\sK\to \sK$ is
surjective if and only if $\epsilon_\sF\colon \pi^*\pi_*\sF\to \sF$ is an
isomorphism. The first claim of \ref{TI:strong:iso} thus follows from
\ref{TI:strong:surj} applied to $\sK$.  The last claim of \ref{TI:strong:iso}
also follows since
$0\to \pi_*\sK\to \sO_X^{\oplus n}\to \pi_*\sF\to 0$ is exact.

If $\sF$ is a vector bundle then (a)$\implies$(b)$\iff$(c) by
\ref{TI:strong:surj}. If (c) holds, then $\varphi\colon
\sO_\stX^{\oplus n}\to \sF$ is a surjective homomorphism of vector bundles of
rank $n$, hence an isomorphism.  It follows that $\pi_*\sF=\sO_X^{\oplus n}$ is
a vector bundle of rank $n$. Then $\epsilon$ is a surjection of vector bundles
of rank $n$, hence an isomorphism.
\end{proof}

\begin{proof}[Proof of Theorem~\ref{T:STRONG-MODULES}]
If $\sF=\pi^*\sG$, then the counit map
$\epsilon \colon \pi^*\pi_*\pi^*\sG\to \pi^*\sG$ is bijective since the
unit map $\eta\colon \sG\to \pi_*\pi^*\sG$ is bijective. The remaining claims
follows from
Theorem~\ref{T:strong-modules:gen}\ref{TI:strong:iso} and \ref{TI:strong:vb}
applied to all points $x\in |X|$.
\end{proof}

As a consequence, we can prove Theorem~\ref{T:MAIN-THEOREM} when $f$ is a closed
immersion.

\begin{corollary}\label{C:strong-immersions}
Let $f\colon \stX\inj \stY$ be a closed immersion of stacks with good moduli
spaces $\pi_X\colon \stX\to X$ and $\pi_Y\colon \stY\to Y$. Then the induced
map $g\colon X\inj Y$ is a closed immersion. If $\sI$ denotes the
ideal sheaf of $f$, then $g$ is given by the ideal sheaf $\sJ=(\pi_Y)_*\sI$.
Let $\epsilon\colon (\pi_Y)^*(\pi_Y)_*\sI\to \sI$ denote the counit map.
Let $x\in |X|$ be a point and
let $x_0\in |\stX|$ be the unique special point above $x$.
If $f$ is finitely presented at $x_0$, then the following are equivalent:
\begin{enumerate}
\item $f$ is strong in a neighborhood of $x$;
\item $\epsilon$ is surjective in a neighborhood of $x$; and
\item $i_{x_0}^*\sI$ is a trivial vector bundle.
\end{enumerate}
When the equivalent conditions hold, then $g$ is finitely presented at $x$.
\end{corollary}
\begin{proof}
(i)$\iff$(ii) follows immediately since $(\pi_Y)^{-1}(X)$ is given by the ideal
which is the image of $\epsilon$.
(ii)$\iff$(iii) follows from Theorem~\ref{T:strong-modules:gen}\itemref{TI:strong:surj}
applied to $\sF=\sI$ which is of finite type. The theorem also shows that
$\sJ=(\pi_Y)_*\sI$ is of finite type, that is, $g$ is finitely presented.
\end{proof}

Note that since $f$ is a closed immersion, the naive cotangent complex
$\tau^{\geq -1}\LL_f = \sI/\sI^2[1]$ is concentrated in degree $-1$ so
$i_{x_0}^*\sI = \coho^{-1}(\LDERF i_{x_0}^* \LL_f)$.

\subsection{Descent for complexes}
We also have a version for Theorem~\ref{T:STRONG-MODULES} for complexes
generalizing~\cite[Thm.~1.3]{nevins_descent}. Recall that a complex
$\sF^\bullet\in \Dqc(\stX)$ is $m$-pseudo-coherent if locally there is a
perfect complex $\sP^\bullet$ and a morphism $\varphi\colon \sP^\bullet \to
\sF^\bullet$ such that $\coho^d(\varphi)$ is an isomorphism for $d>m$
and a surjection for $d=m$~\cite[\spref{08CA}]{stacks-project}. If $\stX$
is noetherian, this simply means that $\sF^\bullet$ is bounded above and
has coherent cohomology in degrees $\geq m$~\cite[\spref{08E8}]{stacks-project}.
We say that $\sF^\bullet$ is
pseudo-coherent if it is $m$-pseudo-coherent for all $m$. In the noetherian
case this simply means $\sF^\bullet\in \Dbcoh(\stX)$.

\begin{theorem}\label{T:strong-modules:complexes}
Let $\stX$ be an algebraic stack with good moduli space $\pi\colon \stX\to X$.
Assume that $\pi$ has affine diagonal.
Let $\sF^\bullet\in \Dqc(\stX)$ be a complex of lisse-\'etale
$\sO_\stX$-modules with quasi-coherent cohomology. Let $x\in |X|$ and let
$x_0\in |\stX|$ be the unique special point above $x$. Let $\epsilon\colon
\LDERF\pi^*\RDERF\pi_*\sF^\bullet\to \sF^\bullet$ denote the counit map.
Let $m$ be an integer. The following are equivalent:
\begin{enumerate}
\item $\sF^\bullet$ is $m$-pseudo-coherent at $x_0$ and $\coho^d(\LDERF
  i_{x_0}^* \sF^\bullet)$ is trivial for every $d\geq m$.
\item $\RDERF\pi_*\sF^\bullet$ is $m$-pseudo-coherent and $\coho^d(\epsilon)$
  is an isomorphism for $d>m$ and surjective for $d=m$
  after restricting to an open neighborhood of $x$.
\end{enumerate}
\end{theorem}
\begin{proof}
Since the good moduli space map $\pi$ has affine diagonal $\pi_*$ is of
cohomological dimension zero so that $\RDERF\pi_*$ is
$t$-exact~\cite[Rmk.~3.5]{alper_good-mod-spaces}.

That (ii)$\implies$(i) is immediate. For the converse, let $k$ denote the
largest integer such that $\coho^k(\LDERF i_{x_0}^* \sF^\bullet)\neq 0$.  Then
there is an open neighborhood of $x_0$ such that $\sF^\bullet$ is acyclic in
degrees $\geq \max(m,k+1)$~\cite[Lem.~4.8]{hall-rydh_perfect-complexes-stacks}.
After restricting to an open neighborhood of $x$, we may thus assume that
$\sF^\bullet$ is acyclic in degrees $\geq \max(m,k+1)$. If $k<m$, then there is
nothing to prove. If $k\geq m$, then $\coho^k(\sF^\bullet)$ is of finite type
and commutes with base change.  By Lemma~\ref{L:surjective} we can, after
passing to an open neighborhood of $x$, find a homomorphism $\varphi\colon
\sO_\stX^{\oplus n}[-k]\to \sF^\bullet$ such that $\coho^k(\varphi)$ is
surjective and $\coho^k(\LDERF i_{x_0}^*\varphi)$ is an isomorphism. In
particular, $\coho^k(\epsilon)$ is surjective, $\coho^k(\RDERF\pi_*\sF^\bullet)$
is of finite type and $\RDERF\pi_*\sF^\bullet$ is acyclic in degrees $\geq
k+1$. If $k=m$, then we are done. If not, let $\sG^\bullet$ be a cone of
$\varphi$. Then $\sG^\bullet$ is $m$-pseudo-coherent and acyclic in degrees
$\geq k$ after passing to an open neighborhood of $x$. Thus,
$\coho^k(\epsilon_{\sF^\bullet})$ is bijective if
$\coho^{k-1}(\epsilon_{\sG^\bullet})$ is surjective.  Moreover, $\coho^d(\LDERF
i_{x_0}^*\sG^\bullet)=\coho^d(\LDERF i_{x_0}^*\sF^\bullet)$ for all $d\leq
k-1$. The result now follows by induction on $k$.
\end{proof}

\begin{corollary}
Let $\stX$ be an algebraic stack with good moduli space $\pi\colon \stX\to X$.
Assume that $\pi$ has affine diagonal.
Let $\sF^\bullet\in \Dqc(\stX)$ be a complex.
Then the following are equivalent:
\begin{enumerate}
\item $\sF^\bullet$ is pseudo-coherent and
  $\coho^d(\LDERF i_{x_0}^* \sF^\bullet)$ is trivial for every integer $d$
  and every special point $x_0\in |\stX|$.
\item $\RDERF\pi_*\sF^\bullet$ is pseudo-coherent and the counit
  $\LDERF\pi^*\RDERF\pi_*\sF^\bullet\to \sF^\bullet$ is a quasi-isomorphism.
\end{enumerate}
Under the equivalent conditions, $\sF^\bullet$ is perfect if and only if
$\RDERF\pi_*\sF^\bullet$ is perfect.
If $X$ is quasi-compact and quasi-separated, then it is enough to consider
closed points $x_0$.
\end{corollary}

\subsection{Remark on non-quasi-separated spaces}
The following remark clarifies the meaning of $i_{x_0}$ when $X$ is not
quasi-separated.

\begin{remark}\label{R:gerbe-representatives}
Let $\pi\colon \stX\to X$ be a good moduli space map and let $x_0\in |\stX|$ be a
special point with image $x\in |X|$.
If $X$ is not quasi-separated, then $\stX$ is not
quasi-separated and it may happen that there is no residual gerbe $i_{x_0}\colon
\stG_{x_0}\inj \stX$. This is merely a notational problem as $\stX$ becomes
quasi-separated \'etale-locally on $X$. Let us say that a \emph{gerbe
  representative} of $x_0$ is a stabilizer-preserving morphism $i_{x_0}\colon \stG\to
\stX$ such that $\stG$ is a quasi-separated fppf-gerbe over a field $k$ with
image $\{x_0\}$.  Equivalently, $\stG$ is the residual gerbe of the unique closed
point of $\stX\times_X \Spec k$ for the induced morphism $\Spec k\to X$. When
$X$ is quasi-separated, we recover the residual gerbe by taking $\Spec
\kappa(x)\to X$.

If $\sF$ is a quasi-coherent sheaf of finite type on $\stX$, then $i_{x_0}^*\sF$ is
a vector bundle on $\stG$. That this vector bundle vanishes or is trivial does
not depend on the choice of gerbe representative $i_{x_0}$.
Similarly, whether $i_{x_0}^*\varphi$ is
injective/surjective/bijective does not depend on the choice of $i_{x_0}$.

If $x_0\colon \Spec k\to \stX$ is an ordinary representative of $x_0$, then
$i_{x_0}^*\sF=0$ if and only if $x_0^*\sF=0$ and similarly for
injectivity/surjectivity/bijectivity of $i_{x_0}^*\varphi$. The vector bundle
$x_0^*\sF$ comes with an action of the stabilizer group of $x_0$. This action is
trivial if and only if $i_{x_0}^*\sF$ is a trivial vector bundle.
\end{remark}

\end{section}


\begin{section}{Special and stabilizer preserving morphisms}\label{S:special}

Let $\map{f}{\stX}{\stY}$ be a morphism between algebraic stacks that admit
good moduli spaces. Recall that $f$ is \emph{special} if it maps special points
to special points. 
If $\stX$ and $\stY$ are of finite type over $Y$, then $f$ is special if and
only if $f$ is \emph{weakly saturated}~\cite[Def.~2.8]{alper_loc-quot-struct}.
We introduce the following pointwise variants of the conditions in
Theorem~\ref{T:MAIN-THEOREM}:

\begin{definition}
Let $\map{f}{\stX}{\stY}$ be a morphism between algebraic stacks with good
moduli spaces. Let $x\in |X|$ be a point and let $x_0\in |\stX|$ be the
corresponding special point. We say that:
\begin{enumerate}
\item $f$ is $(-1)$-strong at $x$ if $\stab(x_0)\to \stab(f(x_0))$ is injective;
\item $f$ is $0$-strong at $x$ if $f(x_0)$ is special and
  $\pi_X^{-1}(x)\to \stY$ is fiberwise stabilizer preserving at $f(x_0)$;
\item $f$ is $1$-strong at $x$ if it is $0$-strong at $x$ and
  the (ind-)vector bundle $\coho^{-1}(\LDERF i_{x_0}^*\LL_f)$ is trivial.
\end{enumerate}
\end{definition}

Let $\rho\colon \stX\to \stY\times_Y X$ be the natural map and $\rho|_x\colon
\stX\times_X \Spec \kappa(x)\to \stY\times_Y \Spec \kappa(x)$ its base
change. Note that a point $y\in |\stY\times_Y X|$ is special if and only if its
image in $|\stY|$ is special. Thus, if $n\in \{-1,0\}$ then $f$
is $n$-strong at $x$ $\iff$ $\rho$ is $n$-strong at $x$ $\iff$ $\rho|_x$ is
$n$-strong at $x$.
Representable morphisms are $(-1)$-strong and closed immersions are $0$-strong.
We conclude that $f$ is $(-1)$-strong (resp.\ $0$-strong) if $\rho|_x$ is
representable (resp.\ a closed immersion).

The main result of this section
(Proposition~\ref{P:key-prop}) gives a reverse implication: if $f$ is
$0$-strong at $x$, then $\rho$ is a closed immersion in an open neighborhood of
$x$. We also prove that if $f$ is $(-1)$-strong at $x$, then
$\rho$ is affine in an open neighborhood of $x$
(Lemma~\ref{L:affineness}).

Corollary~\ref{C:strong-immersions} says that if a closed immersion is $1$-strong
at $x$, then it is strong in an open neighborhood. The general
form of the main theorem (Theorem~\ref{T:main-thm:local}) says the same for
general $f$.

We begin with some preparatory lemmas.

\begin{lemma}\label{L:pointwise-stabilizer-pres}
Let $\map{f}{\stX}{\stY}$ be a morphism between stacks with good moduli spaces
and let $y\in |\stY|$ be a special point. Assume that
\begin{enumerate}
\item the stabilizer group of $y$ is affine (e.g., $\Delta_{\pi_Y}$ is affine), and
\item $f^{-1}(y)$ is reduced.
\end{enumerate}
If $f$ is pointwise stabilizer preserving at every point of $f^{-1}(y)$
then $f$ is fiberwise stabilizer preserving at $y$ and every point of $f^{-1}(y)$
is special.
\end{lemma}
\begin{proof*}
  The question is local on $X$ so we can assume that $X$ is affine so that
  $\stX$ is cohomologically affine. We can also replace $Y$ with $\Spec
  \kappa(y)$ and $\stY$ with the residual gerbe $\stG_y$ and $\stX$ with
  $f^{-1}(y)$. Then $f$ is representable, because it is pointwise stabilizer
  preserving. Moreover, $f$ is cohomologically affine, because $\stX$ is cohomologically
  affine and $\stY$ has affine diagonal. It follows that $f$ is
  affine by Serre's
  criterion~\cite[Props.~3.3 and 3.10 (vii)]{alper_good-mod-spaces}.

  Since $f$ is representable and separated, the map between inertia
  $\varphi\colon I_{\stX}\to f^*I_\stY$ is a closed immersion. Moreover,
  $\varphi$ is surjective, hence a nil-immersion, since $f$ is pointwise
  stabilizer preserving at every point. Since $\stY$ is a gerbe,
  $f^*I_\stY$ is flat over
  $\stX$, which is reduced. Thus,
  every (weakly) associated point of $f^*I_\stY$ lies over a generic
  point of $\stX$. Over these points, $\varphi$ is an isomorphism by
  assumption, so
  $\varphi$ is an isomorphism.

  The final claim follows from the following lemma which is due to Daniel Bergh.
\end{proof*}

\begin{lemma}\label{L:stab-pres-over-gerbe}
Let $\map{f}{\stX}{\stY}$ be a morphism between quasi-separated algebraic
stacks. If $\stY$ is an fppf-gerbe over an algebraic space $Y$ and $f$ is
stabilizer preserving, then $\stX$ is an fppf-gerbe over an algebraic space $X$
and there is a cartesian diagram:
\[
\xymatrix{%
\stX\ar[r]^f\ar[d] & \stY\ar[d] \\
X\ar[r]^g & Y.}
\]
\end{lemma}
\begin{proof}
Since $I_{\stX}=f^*I_\stY$ is flat and of finite presentation, $\stX$ is a
gerbe over an algebraic space $X$. The structure map $\stX\to X$ is initial
among maps to algebraic spaces so we have a commutative diagram as above.
The map $\stX\to \stY\times_Y X$ is a faithfully flat quasi-compact monomorphism,
hence an isomorphism.
\end{proof}

\begin{remark}
In Lemma~\ref{L:pointwise-stabilizer-pres}, it is not enough that $f$ is
pointwise stabilizer preserving at every \emph{special} point of
$f^{-1}(y)$. Indeed, consider $\Gm$ acting by multiplication on $\AA^1$. Then
$f\colon [\AA^1/\Gm]\to B\Gm$ is smooth and pointwise stabilizer preserving at
the unique special point but not stabilizer preserving. It is also not enough
that $f$ is pointwise stabilizer preserving at every point of
$\pi_X^{-1}(x)$ to deduce that $\pi_X^{-1}(x)\to \stY$ is
fiberwise stabilizer preserving at $y$.  Indeed, if $f\colon [\AA^1/\Gmu_2]\to
B\Gmu_2$ where $\Gmu_2$ acts by multiplication, then $f$ is smooth and pointwise
stabilizer preserving at the origin but $\pi_X^{-1}(0)=
\bigr[\bigl(\Spec k[x]/(x^2)\bigr) / \Gmu_2\bigr] \to B\Gmu_2$ is not
stabilizer preserving. The latter map also shows that it is necessary that
$f^{-1}(y)$ is reduced in Lemma~\ref{L:pointwise-stabilizer-pres}.
\end{remark}

\begin{lemma}\label{L:affineness}
Let $\map{f}{\stX}{\stY}$ be a morphism between algebraic stacks with good moduli
spaces. Consider the induced map $\map{\rho}{\stX}{\stY\times_Y X}$.
\begin{enumerate}
\item If either $\stX$ is noetherian or $f$ is locally of finite type,
  then $\rho$ is of finite type.
\item If $\Delta_{\pi_X}$ and $\Delta_{\pi_Y}$ are affine and
  $f$ is $(-1)$-strong at $x\in |X|$, then
  $\rho|_U\colon \stX\times_X U\to \stY\times_Y U$ is affine for some
  open neighborhood $U$ of $x$.
\item $f$ is $0$-strong at $x\in |X|$ if and only if $w_0:=\rho(x_0)$ is special,
  and $\rho^{-1}(\stG_{w_0})\to \stG_{w_0}$ is an isomorphism.
\end{enumerate}
\end{lemma}
\begin{proof}
  Recall that if $\stX$ is noetherian, then $\stX\to X$ is of finite type.
  In either case, we thus deduce that $\rho$ is locally of finite type.
  Since $\pi_X$ and $\pi_Y$ are quasi-compact and quasi-separated, so is $\rho$.

  For
  (ii), we may assume that $X$ and $Y$ are affine. Then $\stX$ is cohomologically
  affine and $\Delta_\stY$ is affine and the result follows
  from~\cite[Prop.~12.5 (1)]{alper-hall-rydh_etale-local-stacks} when
  $x$ is closed. For non-closed $x$, we can instead conclude that
  $\stX\times_X \Spec \sO_{X,x}\to \stY$ is affine and then that
  $\stX\times_X U \to \stY$ is affine for some open neighborhood $U$ by
  \cite[Thm.~C]{rydh_noetherian-approx}.

  For (iii), first note that $f$ is $0$-strong at $x$ if and only if $\rho|_x$
  is $0$-strong at $x$. We can thus assume that $X=Y=\Spec \kappa(x)$ and
  that $\rho=f$. Then $x_0\in |\stX|$ and $w_0\in |\stY|$ are closed.
  Let $\stX_{w_0}=f^{-1}(\stG_{w_0})$. Since $\stX_{w_0}\inj \stX$ is a closed
  immersion, the good moduli space of $\stX_{w_0}$ coincides with $X=Y=\Spec
  \kappa(x)$. If $\stX_{w_0}\to \stG_{w_0}$ is stabilizer preserving, then
  $\stX_{w_0}$ is a gerbe over $\Spec \kappa(x)$ and $\stX_{w_0}\to \stG_{w_0}$
  is an isomorphism (Lemma~\ref{L:stab-pres-over-gerbe}). The converse is
  obvious.
\end{proof}

The following proposition is a key result in the paper that is used to reduce
the main theorem from arbitrary maps to closed immersions so that we can use
the results of the previous section.

\begin{proposition}\label{P:key-prop}
Let $\map{f}{\stX}{\stY}$ be a morphism between algebraic stacks with good
moduli spaces. Consider the map $\map{\rho}{\stX}{\stY\times_Y X}$.
Assume that $\pi_X$ and $\pi_Y$ have affine diagonals and that
either $\stX$ is noetherian or $f$ is locally of finite type. Let
$U\subseteq |X|$ be the set of points $x\in |X|$ such that
$f$ is $0$-strong at $x$. Then
\begin{enumerate}
\item $U$ is open in $X$.
\item $\rho|_U\colon \stX\times_X U\to \stY\times_Y U$ is a closed immersion.
\end{enumerate}
\end{proposition}
\begin{proof}
Let $x\in U$. After replacing $X$ with an open neighborhood of $x$, we can
assume that $\rho$ is affine and of finite type (Lemma~\ref{L:affineness}).
We can also replace $\stY$ with $\stY\times_Y X$ and assume that $X=Y$.
Let $y_0=f(x_0)$ which is special. Then $f^{-1}(\stG_{y_0})\to \stG_{y_0}$
is an isomorphism (Lemma~\ref{L:affineness}). In particular, $f$ is
quasi-finite at $x_0$.
After replacing $X=Y$ with an open neighborhood, we can thus assume that
$f$ is quasi-finite and affine.

By Zariski's main theorem~\cite[Thm.~8.1]{rydh_approximation-sheaves}, $f$
factors as an open immersion $j\colon \stX\to \stY'$ followed by a finite
morphism $g\colon \stY'\to \stY$. Note that $\stY'$ has a good moduli space
$Y'$ which is affine over $Y$. We may thus replace $\stY'$ with
$\stY'\times_{Y'} X$ so that $X=Y'=Y$. Then $j(x_0)$ is the unique special
point above $x$ so $j|_x$ is an isomorphism.
After replacing $X=Y$ with the open neighborhood
$Y'\smallsetminus \pi_{Y'}(\stY'\smallsetminus \stX)$ we can thus assume
that $\stX=\stY'$ so that $f$ is finite.

Let us finally prove that $f|_U$ is a closed immersion for an open
neighborhood $U$ of $x$. Since $\sO_\stY\to f_*\sO_{\stX}$ is
finite, the cokernel $\sF$ is of finite type. Since
$f^{-1}(\stG_{y_0})\to \stG_{y_0}$
is an isomorphism, we have that $y_0\notin \Supp \sF$.
We now take $U$ to be the complement
of the closed subset $\pi_Y(\Supp \sF)$.
\end{proof}

\begin{remark}
If $f$ is special at $x_0$ and \emph{pointwise} stabilizer preserving at every
point in the fiber $f^{-1}(f(x_0))$, then a similar argument shows that $\rho|_U$
is finite.
\end{remark}

\begin{remark}\label{R:0-strong-fiberwise-stab-pres}
A consequence of Proposition~\ref{P:key-prop} is that if $f$ is $0$-strong at
$x$, then $f$ is stabilizer preserving at $x_0$, that is, $f$ is stabilizer
preserving after restricting to an open neighborhood $U$ of $x$.
\end{remark}

\end{section}


\begin{section}{Cotangent complexes}\label{S:cotangent-complexes}

Let $\stX$ be an algebraic stack with a good moduli space and let $x\colon
\Spec k\to \stX$ be a point. Coherent sheaves on the residual gerbe $\stG_x$
are vector bundles, that is, locally free sheaves of finite rank.
Quasi-coherent sheaves on $\stG_x$ are ind-vector bundles, that is, union of vector
bundles. We say that an ind-vector bundle is \emph{trivial} if it is free as a
quasi-coherent sheaf. An ind-vector bundle $\sF$ on $\stG_x$ gives rise to a
$k$-vector space $V=x^*\sF$ with a $\stab(x)$-action. The bundle $\sF$ is
trivial, if and only if the representation $V$ is trivial.

Quotients and sub-bundles of trivial ind-vector bundles on $\stG_x$ are trivial.
If $x$ is special, then $\stab(x)$ is linearly reductive and extensions of
trivial bundles are trivial.

\begin{definition}
Let $\map{f}{\stX}{\stY}$ be a morphism of algebraic stacks
with good moduli spaces. Let $x_0\in |\stX|$ be a (special) point and let
$n\geq -1$ be an integer. We say that $f$ is \emph{$n$-$\LL$-strong} at $x_0$ if
the ind-vector bundle $\coho^d(\LDERF i_{x_0}^*\LL_f)$ is trivial for every
$d\geq -n$. Equivalently, if $x_0\colon \Spec k\to \stX$ is a representative of
the point, then the $\stab(x_0)$-action on the $k$-vector space $H^d(\LDERF
x_0^*\LL_f)$ is trivial for every $d\geq -n$.
We say that $f$ is \emph{$\LL$-strong} at $x_0$ if $f$ is $n$-$\LL$-strong at
$x_0$ for all $n$.
\end{definition}

If $f$ is locally of finite type, then 
$\coho^0(\LDERF i_{x_0}^*\LL_f)=i_{x_0}^*\Omega_f$ is finite-dimensional.
If $f$ is locally of finite presentation, then one can show that
$\coho^{-1}(\LDERF i_{x_0}^*\LL_f)$ is finite-dimensional. We will not need this.

\begin{lemma}\label{L:L-strong}
Let $x\in |X|$ be a point and let $x_0\in |\stX|$ be the corresponding
special point.
\begin{enumerate}
\item\label{LI:L-strong:n-strong}
  Let $n\in \{-1,0,1\}$. If $f$ is $n$-strong at $x$, then $f$
  is $n$-$\LL$-strong at $x_0$.
\item\label{LI:L-strong:strong}
  If $f$ is strong at $x$, then $f$ is $1$-strong at $x$.
\item\label{LI:L-strong:flat}
  Assume that $\pi_X$ and $\pi_Y$ have affine diagonals and 
  that either $\stX$ is noetherian or $f$ is locally of finite type.
  If $f$ is flat and $0$-strong at $x$, then $f$ is $\LL$-strong at $x_0$.
  In particular, $f$ is $1$-strong at $x$.
\end{enumerate}
\end{lemma}
\begin{proof}
Let $g\colon X\to Y$ denote the induced morphism on good moduli spaces.

(i) For $n=-1$, if $I_\stX\to f^*I_\stY$ is injective at $x_0$, then $f$
is relatively Deligne--Mumford in an open neighborhood of $x_0$ and hence $\LL_f$
is concentrated in degrees $\leq 0$ in that neighborhood.
For $n=0$, we have that $i_{x_0}^*\Omega_\rho=0$ by Lemma~\ref{L:affineness}
so $i_{x_0}^*(\pi_X)^*\Omega_g \to i_{x_0}^*\Omega_f$ is surjective
which shows that the latter is trivial.
For $n=1$, we have that $\coho^{-1}(\LDERF i_x^*\LL_f)$ is trivial by
definition.

(ii) Strong morphisms are $0$-strong so it is enough to prove that $f$ is
$1$-$\LL$-strong at $x_0$.
We have a natural map $\LDERF (\pi_X)^*\LL_g \to \LL_f$. Let $E$ denote
the cone of this map. If $f$ is strong, that is, if $\stX=\stY\times_Y X$,
then $E$ is the cotangent complex measuring the difference between the
fiber product in derived algebraic geometry and classical algebraic geometry.
The complex $E$ is concentrated in degrees
$\leq -2$ \cite[\spref{09DM}]{stacks-project}.
It follows that
\[
\coho^d(\LDERF i_{x_0}^*\LDERF (\pi_X)^*\LL_g)\to \coho^d(\LDERF i_{x_0}^*\LL_f)
\]
is an isomorphism for $d=0$, that is, we have an equality of cotangent bundles,
and surjective for $d=-1$. The result follows.

(iii) Since $f$ is $0$-strong, we have that $y_0=f(x_0)$ is closed and we may
assume that $f$ is stabilizer preserving
(Remark~\ref{R:0-strong-fiberwise-stab-pres}).
Since $f$ is flat, the cotangent complex commutes with arbitrary pull-back
and we may assume that $\stY=\stG_{y_0}$.
Then $\stY\to Y$ is flat and $f$ is strong (Lemma~\ref{L:stab-pres-over-gerbe}).
By flat base change we then have that $\LL_f = \LDERF (\pi_X)^* \LL_g$
and that $\LDERF i_{x_0}^*\LL_f$ is the pull-back of $\LDERF i_x^*\LL_g$
along $\stG_{x_0}\to \Spec \kappa(x)$ and similarly for the $d$th cohomology
sheaf.
\end{proof}

\begin{lemma}\label{L:L-strong-comp}
Let $f\colon \stX\to \stY$ and $g\colon \stY\to \stZ$ be morphisms between
algebraic stacks with good moduli spaces.
Let $x_0\in |\stX|$ be a special
point and $y_0=f(x_0)$. Let $n\geq -1$ be an integer.
\begin{enumerate}
\item If $f$ is $n$-$\LL$-$strong$ at $x_0$ and $g$ is
  $n$-$\LL$-strong at $y_0$, then $g\circ f$ is $n$-$\LL$-$strong$ at $x_0$.
\item If $g\circ f$ is $n$-$\LL$-$strong$ at $x$ and $g$ is
  $(n-1)$-$\LL$-strong at $y_0$, then $f$ is $n$-$\LL$-$strong$ at $x_0$.
\end{enumerate}
\end{lemma}
\begin{proof}
The fundamental triangle for the composition $g\circ f$ gives the long
exact sequence
\vspace{-3mm}
\[
\setlength\arraycolsep{2pt}
\begin{matrix}
 &&& \dots &\longrightarrow & \coho^{-(n+1)}(\LDERF i_{x_0}^*\LL_f) & \longrightarrow \\
 \longrightarrow & \coho^{-n}(\LDERF (f\circ i_{x_0})^*\LL_g) & \longrightarrow
     & \coho^{-n}(\LDERF i_{x_0}^*\LL_{gf})       & \longrightarrow
     & \coho^{-n}(\LDERF i_{x_0}^*\LL_f)          & \longrightarrow \\
 \longrightarrow & \coho^{-(n-1)}(\LDERF (f\circ i_{x_0})^*\LL_g) & \longrightarrow
     & \dots
\end{matrix}
\]
Since the stabilizer of $x_0$ is linearly reductive, extensions of trivial
ind-vector bundles on $\stG_{x_0}$ are trivial. The result follows.
\end{proof}

\begin{proposition}\label{P:1-strong-comp}
Let $f\colon \stX\to \stY$ and $g\colon \stY\to \stZ$ be morphisms between
algebraic stacks with good moduli spaces.
Assume that $\pi_X$, $\pi_Y$ and $\pi_Z$ have affine diagonals and 
that either $\stX$ and $\stY$ are noetherian or that
$f$ and $g$ are locally of finite type.
Let $x\in |X|$ be a
point and $y=f(x)$. Let $n\in \{0,1\}$.
\begin{enumerate}
\item If $f$ is $n$-strong at $x$ and $g$ is $n$-strong at $y$, then
  $g\circ f$ is $n$-strong at~$x$.
\item If $g\circ f$ is $n$-strong at $x$ and $g$ is $(n-1)$-strong at $y$, then
  $f$ is $n$-strong at~$x$.
\end{enumerate}
\end{proposition}
\begin{proof}
When $n=0$, the result follows directly from Lemma~\ref{L:affineness} and
Proposition~\ref{P:key-prop}. For $n=1$, the result then follows from
Lemmas~\ref{L:L-strong}\ref{LI:L-strong:n-strong} and
\ref{L:L-strong-comp}.
\end{proof}

\end{section}


\begin{section}{Proof of the main theorem}\label{S:pf-of-main-thm}
In this section, we prove a local generalization of Theorem~\ref{T:MAIN-THEOREM}
and prove that all the listed properties of $f$ descend to $g$ except for the
properties concerning local complete intersections and regular
immersions that we defer to the following sections.

\begin{lemma}\label{L:finite-type}
Let $\map{f}{\stX}{\stY}$ be a morphism between algebraic stacks with good moduli
spaces. Assume that $\pi_X$ and $\pi_Y$ have affine diagonals.
Further assume that $f$ is $0$-strong at $x\in |X|$.
If $f$ is locally of finite type at $x_0$, then $\map{g}{X}{Y}$ is
locally of finite type at $x$.
\end{lemma}
\begin{proof}
After replacing $X$ with an open neighborhood of $x$, we may assume that
$f$ is locally of finite type. After replacing $X$ with a further open
neighborhood, we can assume that $\rho\colon \stX\to \stY\times_Y X$ is
a closed immersion by Proposition~\ref{P:key-prop}.
We may also assume that $X=\Spec B$ and $Y=\Spec A$ are affine.

Write $B$ as the
union of its finitely generated subalgebras $B_\lambda$ and let
$X_\lambda=\Spec B_\lambda$.
Since $\rho\colon \stX\to \varprojlim_\lambda \stY\times_Y X_\lambda$ is a closed immersion and $f$ is of finite type, it follows that
$\rho_\lambda\colon \stX\to \stY\times_Y X_\lambda$ is a closed immersion for
sufficiently large
$\lambda$.
Hence the induced map on good moduli spaces $X\to X_\lambda$ is a closed immersion
as well and we conclude that $g$ is of finite type.
\end{proof}

\begin{theorem}\label{T:main-thm:local}
Let $\map{f}{\stX}{\stY}$ be a morphism between algebraic stacks with good moduli
spaces. Assume that $\pi_X$ and $\pi_Y$ have affine diagonals.
Let $x\in |X|$ be a point. Assume that either $\stX$ and $\stY$ are
locally noetherian
or that $f$ is locally of finite presentation at $x_0$.
Then the following are equivalent
\begin{enumerate}
\item $f$ is $1$-strong at $x$,
\item $f$ is strong at $x$.
\end{enumerate}
\end{theorem}

\begin{proof}
That (ii)$\implies$(i) is Lemma~\ref{L:L-strong}. For the converse,
we can assume that $X$ and $Y$ are affine. After replacing $X$ with an
open neighborhood of $x$, we can assume that
$\rho\colon \stX\to \stY\times_Y X$ is a closed immersion
(Proposition~\ref{P:key-prop}).

If $\stY$ is locally noetherian, then $\pi_Y$ is of finite type so
$\stY\times_Y X$ is noetherian and $\rho$ is of finite presentation.
If $f$ is locally of finite presentation at $x_0$,
then $g$ is locally of finite type at $x$
by Lemma~\ref{L:finite-type} which implies that $\rho$ is of finite presentation
at $x_0$.

Moreover, $\rho$ is $1$-strong at $x$ by Proposition~\ref{P:1-strong-comp}.  We
can thus conclude that $\rho$, and hence $f$, is strong at $x$ by
Corollary~\ref{C:strong-immersions}.
\end{proof}

That most properties in Theorem~\ref{T:MAIN-THEOREM} descend follows
from purity of $\pi_Y$.

\begin{proposition}\label{P:main-thm:properties}
Let $\map{f}{\stX}{\stY}$ be a morphism between algebraic stacks with good
moduli spaces. Suppose that $f$ is strong.
If $f$ has one of the properties:
\begin{enumerate}
\item locally of finite type, locally of finite presentation, flat,
\item syntomic, regular, smooth, \'etale, representable,
  monomorphism and locally of finite type, unramified, locally quasi-finite,
\item affine, closed immersion, immersion, open immersion
\item finite, proper, separated;
\end{enumerate}
then so does $g\colon X\to Y$.
\end{proposition}
\begin{proof}
Recall that $\pi_Y\colon \stY\to Y$ is \emph{pure} (Section~\ref{S:gms}).
If $f$ is locally of finite type, locally of finite presentation, or flat, then
so is $g$ by descent along
$\pi_Y$~\cite[\spref{08XD}, \spref{08XE}]{stacks-project}.

The properties in (ii) are combinations of flat, locally of finite type,
locally of finite presentation and conditions on the geometric fibers so
they descend.

If $f$ is affine, then so is $g$ since $X=\Spec_Y (\pi_Y)_* f_*\sO_\stX$.  If
$f$ is a closed immersion, then so is $g$ since $(\pi_Y)_*$ is exact.  If $f$
is an immersion, let $x\in |X|$ and let $U$ be an open neighborhood of $x_0\in
|\stY|$ such that $f|_U$ is a closed immersion. Then $g|_{Y\smallsetminus
  \pi_Y(\stY\smallsetminus U)}$ is a closed immersion.
An open immersion is the same thing as an \'etale monomorphism.

If $f$ is universally closed, then so is $g$ since $\pi_X$ and $\pi_Y$ are
surjective and universally closed.
If $f$ is separated, then $\Delta_f$ is a closed immersion
and it follows that $\Delta_g$ is a closed immersion, that is, $g$ is separated.
We conclude that if $f$ is proper, then so is $g$.
Finite is equivalent to proper and affine and thus also descend (or use
\cite[\spref{08XD}]{stacks-project}).
\end{proof}

\begin{proof}[Proof of Theorem~\ref{T:MAIN-THEOREM}]
Conditions (i)--(iii) says that $f$ is $1$-strong at every point $x\in |X|$.
Thus, if they are satisfied, then $f$ is strong by
Theorem~\ref{T:main-thm:local}. Conversely, if $f$ is strong, then $f$
is special and stabilizer preserving and also $1$-$\LL$-strong at every
special point by Lemma~\ref{L:L-strong}\ref{LI:L-strong:strong}.

If $f$ has one of the properties of (a)--(c) then so does $g$ by
Propositions~\ref{P:main-thm:properties} (most properties),
\ref{P:main-thm:regular-immersion} (quasi-regular and Koszul-regular)
and \ref{P:main-thm:lci} (lci).

If $\pi_Y$ is a coarse moduli space, then every point of $\stY$ is special so
(i) is redundant.  If $f$ has reduced fibers, then fiberwise stabilizer
preserving at $y_0\in |\stY|$ is equivalent to pointwise stabilizer preserving
at every point of $f^{-1}(y_0)$ by Lemma~\ref{L:pointwise-stabilizer-pres}.
Conditions (i)--(ii) says that $f$ is $0$-strong at every point $x\in |X|$.  If
in addition $f$ is flat, then $f$ is $1$-strong at every point $x\in |X|$ by
Lemma~\ref{L:L-strong}\ref{LI:L-strong:flat}.
\end{proof}

\begin{proof}[Proof of Corollary~\ref{C:MAPS-TO-STACKS}]
It is immediate that (i)$\implies$(ii)$\implies$(iii)$\implies$(iv).  If
$g_1,g_2\colon X\to \stY$ are two maps through which $f$ factors, then the we
obtain a map $\tau\colon \stX\to \stY\times_{\stY\times \stY} X$. Since $\pi$
is initial among algebraic spaces, $\tau$ factors uniquely as $\stX\to X\to
\stY\times_{\stY\times \stY} X$, that is, we have a unique $2$-isomorphism
between $g_1$ and $g_2$.

To see that (iv)$\implies$(i), we may thus work fppf-locally on $X$.  Let $U\to
\stY$ be a smooth presentation and let $\stX'=\stX\times_\stY U$. Then
$\stX'\to \stX$ is smooth, surjective and fiberwise stabilizer preserving by
(iv).
By Theorem~\ref{T:MAIN-THEOREM}, it thus descends to a smooth surjective
morphism $X'\to X$. Since $\stX'\to X'$ is a coarse moduli space, we obtain a
factorization $\stX'\to X'\to U$ which gives a map $X'\to U\to \stY$ as
requested.
\end{proof}
\end{section}


\begin{section}{Regular immersions}\label{S:reg-emb}
In this section, we prove the part about regular immersions in
Theorem~\ref{T:MAIN-THEOREM}. We begin with recalling various notions of
regular sequences and regular immersions in the non-noetherian situation. The
reader who is only interested in the noetherian situation can skip ahead to
\S\ref{S:descent-regular-immersions}.

\subsection{Regular sequences}
Let $A$ be a ring, let $f_1,f_2,\dots,f_n\in A$ be a sequence of elements and let
$I=(f_1,f_2,\dots,f_n)$. There are three slightly different notions of
$f_1,f_2,\dots,f_n$ being a regular sequence~\cite[0.15, 16.9, 19.5]{egaIV}, \cite[Exp.~VII]{sga6},
\cite[\spref{062D}, \spref{07CU}, \spref{067M}]{stacks-project}.
\begin{enumerate}
\item $f_i$ is a non-zero divisor in $A/(f_1,\dots,f_{i-1})$ for all
  $i=1,2,\dots,n$.
\item The Koszul complex $K_\bullet(A,f_1,f_2,\dots,f_n)$ is acyclic in degrees
  $>0$.
\item $I/I^2$ is locally free of rank $n$ and the canonical map
  $\Sym_{A/I}(I/I^2)\to \bigoplus_{d\geq 0} I^d/I^{d+1}$ is an isomorphism.
\end{enumerate}
We say that the sequence $(f_i)$ is \emph{regular} in (i), \emph{Koszul-regular} in (ii) and
\emph{quasi-regular} in (iii). If $(f_1,f_2,\dots,f_n)=(g_1,g_2,\dots,g_n)$,
then $\{f_i\}$ is Koszul-regular if and only if $\{g_i\}$ is
Koszul-regular~\spcite{066A}.  This is also trivially true for quasi-regularity
but false for regular sequences in general~\cite[Rmq.~16.9.6 (ii)]{egaIV}.

\subsection{Schemes}
Let $X$ be a scheme and let $f\colon Z\inj X$ be a closed immersion with
corresponding sheaf of ideals $\sI$. We say that $f$ is a regular (resp.\ a
Koszul-regular, resp.\ a quasi-regular) immersion at $x\in |X|$ if there exists
an affine open neighborhood $U=\Spec A\subseteq X$ of $x$ such that $\sI$ is
generated by a regular (resp.\ Koszul-regular, resp.\ quasi-regular) sequence.
We have that regular $\implies$ Koszul-regular $\implies$ quasi-regular.  When
$X$ is locally noetherian, all three notions coincide,
see \cite[Prop.~16.9.11, Cor.~19.5.2]{egaIV} or \cite[\spref{063E}, \spref{063I}]{stacks-project}.
Koszul-regularity and quasi-regularity are easily seen to be
fpqc-local~\spcite{068N} whereas regularity is not. Indeed, any Koszul-regular
immersion is smooth-locally regular~\spcite{068Q}.

\subsection{Algebraic stacks}
We say that a closed immersion $f\colon \stZ\inj \stX$ of algebraic stacks is
Koszul-regular (resp.\ quasi-regular) at a point $x\in |\stX|$ if there exists
a scheme $U$, a smooth morphism $U\to \stX$ and a point $u\in |U|$ above $x$,
such that $\stZ\times_\stX U\to U$ is Koszul-regular (resp.\ quasi-regular)
at $u$.
This then holds for any $U$ and any point $u$ above $x$.

Let $f\colon \stZ\inj \stX$ be a closed immersion, let $x\in |\stX|$ be a point
and let $i_x\colon \stG_x\inj \stX$ denote the inclusion of the residual gerbe
(or a gerbe representative as in Remark~\ref{R:gerbe-representatives}).
Consider the following conditions.
\begin{enumerate}
\item $f$ is Koszul-regular at $x$,
\item $\sI/\sI^2$ is locally free at $x$ and $\LL_f\simeq (\sI/\sI^2)[1]$ in a
  neighborhood of $x$.
\item $\coho^{-n}(\LDERF i_x^*\LL_f)=0$ for $n\geq 2$, and
\item $\coho^{-n}(\LDERF i_x^*\LL_f)=0$ for $n=2$.
\end{enumerate}
Then (i)$\implies$(ii)~\spcite{08SK} and trivially
(ii)$\implies$(iii)$\implies$(iv). If $\stX$ is locally noetherian, then all
conditions are equivalent,
see~\cite[Thm.~6.25]{andre_homologie-algebres-commutatives}.

\subsection{Descent of regular immersions}\label{S:descent-regular-immersions}
Recall that a closed immersion $f\colon \stX\to \stY$ with ideal sheaf $\sI$
is a quasi-regular immersion, if
\begin{enumerate}
\item $\sI$ is finitely generated (i.e., $f$ is finitely presented),
\item the conormal sheaf $\sN_f:=\sI/\sI^2=f^*\sI$ is a vector bundle on
  $\stX$, and
\item the natural map $\Sym_{\sO_X}(\sN_f)\to \bigoplus_{d\geq 0}
  \sI^d/\sI^{d+1}$ is an isomorphism.
\end{enumerate}

\begin{proposition}\label{P:main-thm:regular-immersion}
Let $f\colon \stX\to \stY$ be a closed immersion between stacks with
good moduli spaces and let $g\colon X\to Y$ denote the induced closed immersion
of good moduli spaces. Suppose that $f$ is strong.
\begin{enumerate}
\item If $f$ is a quasi-regular immersion, then so is $g$.
\item If $f$ is Koszul-regular, then so is $g$.
\end{enumerate}
In either case, the adjunction maps $(\pi_X)^*\sN_g\to \sN_f$
and $\sN_g\to (\pi_X)_*\sN_f$ are isomorphisms.
\end{proposition}
\begin{proof}
Let $\sI$ denote the ideal sheaf of $f$ and $\sJ$ the ideal sheaf of $g$.
Since $f$ is strong, we have that $(\pi_Y)^*\sJ^d\to \sI^d$ is surjective
for all $d\geq 0$. Thus, since $f$ is cohomologically affine,
$\sJ^d=(\pi_Y)_*(\sI^d)$ for all $d\geq 0$ and hence
$\sJ^d/\sJ^{d+1}=(\pi_X)_*(\sI^d/\sI^{d+1})$. In particular,
\[
\sN_g=(\pi_X)_*\sN_f,\quad\text{and}\quad
\bigoplus_{d\geq 0} \sJ^d/\sJ^{d+1}
 = (\pi_X)_*\bigoplus_{d\geq 0} \sI^d/\sI^{d+1}.
\]
If $\sN_f$
is a vector bundle, then so is $(\pi_X)_*\sN_f$ by
Theorem~\ref{T:strong-modules:gen}\itemref{TI:strong:vb} applied to
$\sF=\sN_f=f^*\sI$ along $\pi_X$, and
$(\pi_X)^*(\pi_X)_*\sN_f\to \sN_f$ is an isomorphism. It follows that
$\Sym_{\sO_X}(\sN_g)=(\pi_X)_*\Sym_{\sO_\stX}(\sN_f)$. Thus, if $f$ is
quasi-regular, then so is $g$.

Now, suppose that $f$ is Koszul-regular.
To see that $g$ is Koszul-regular we may work locally
on $Y$ and assume that $Y=\Spec A$ is affine. Choose a sequence of elements
$f_1,f_2,\dots,f_n\in \sJ$ that gives a basis of $\sN_g$. After replacing
$Y$ with an open neighborhood, we may assume that $\sJ=(f_1,f_2,\dots,f_n)$.
Then $(\pi_Y)^*f_1,\dots,(\pi_Y)^*f_n\in \sI$ generates $\sI$ and gives a basis
of $\sN_f=(\pi_X)^*\sN_g$. Consider the Koszul complex
$K_\bullet=K_\bullet(A,f_1,f_2,\dots,f_n)$.
Since $f$ is Koszul-regular, the pull-back $(\pi_Y)^*K_\bullet$ is acyclic in
degrees $>0$. Since $(\pi_Y)_*$ is exact, the push-forward
$(\pi_Y)_*(\pi_Y)^*K_\bullet=K_\bullet$ is also
acyclic in degrees $>0$ and we conclude that $g$ is Koszul-regular.
\end{proof}

\begin{remark}[cf.\ {\cite[Thm.~2.2 (ii)]{edidin_strong-regular-embeddings}}]
Suppose that $\stX$ is noetherian (or merely quasi-compact and quasi-separated).
If $f$ is a strong quasi-regular immersion, then there exists a finite
stratification $X=\bigcup X_k$ by locally closed subspaces such that
$(\sN_f)|_{\stX\times_X X_k}$ is a
trivial vector bundle for every $k$. Indeed, there exists a finite
stratification of $X$ in schemes. After
refining the stratification we may assume that
$\sN_g|_{X_k}$ is a trivial
vector bundle for every $k$ and the result follows.
\end{remark}

\end{section}


\begin{section}{Local complete intersection morphisms}\label{S:lci}
In this section, we prove the part about lci morphisms in
Theorem~\ref{T:MAIN-THEOREM}.
We begin with recalling the definition of lci morphisms $f\colon \stX\to
\stY$ when either (a) $\stX$ and $\stY$ are noetherian schemes or stacks but
$f$ is not locally of finite type, or (b) $\stX$ and $\stY$ are not necessarily
noetherian but $f$ is locally of finite presentation. The reader who only is
interested in morphisms of finite type between noetherian stacks can skip
ahead to \S\ref{S:descent-lci}.
Main references for this section are~\cite{avramov_lci},
\spcite{068E} and \cite[19.3]{egaIV}.

\subsection{Algebraic schemes}
Let $X$ be a scheme of finite type over a field $k$. We say that $X$ is a
\emph{local complete intersection at $x$}, abbreviated lci, if locally around
$x$, there is a closed immersion $i\colon X\inj Y$ where $Y$ is smooth over $k$
and $i$ is regular at $x$. If $X$ is lci at $x$, then $i$
is regular at $x$ for any such factorization. Equivalently, $H^2(\LL_{X/k}\otimes
\kappa(x))=0$ and then $H^{-n}(\LL_{X/k}\otimes \kappa(x))=0$ for all $n\geq 2$.

\subsection{Morphisms locally of finite presentation}\label{SS:lci:fp}
If $f\colon \stX\to \stY$ is a morphism, locally of finite
presentation, between algebraic stacks, then we say that $f$ is lci at $x\in
|\stX|$ if there exists a commutative diagram
\[
\xymatrix{
U\ar[d]\ar@{(->}[r]^i & V\ar[d]\\
\stX\ar[r]^f & \stY
}
\]
where $U$ and $V$ are schemes, the vertical maps are smooth and $i$ is a closed immersion such 
that $i$ is Koszul-regular at a point $u$ above $x$. Then $i$ is Koszul-regular
for any such diagram~\cite[\spref{0692}, \spref{069P}]{stacks-project}.

A closed immersion is Koszul-regular if and only if it is lci.  Let $f\colon
\stX\to \stY$ be a morphism, locally of finite presentation, let $x\in |\stX|$
and consider the following conditions:
\begin{enumerate}
\item $f$ is lci at $x$,
\item $\LL_f|_U$ is perfect of Tor-amplitude $[-1,1]$ in an open neighborhood $U$
  of $x$.
\item $\coho^{-n}(\LDERF i_x^*\LL_f)=0$ for $n\geq 2$, and
\item $\coho^{-n}(\LDERF i_x^*\LL_f)=0$ for $n=2$.
\end{enumerate}
Then (i)$\implies$(ii)$\implies$(iii)$\implies$(iv) and all conditions are
equivalent if $\stX$ is locally noetherian. This is an easy consequence of
the corresponding result for Koszul-regular immersions using that the cotangent
complex of a smooth morphism is perfect of Tor-amplitude $[0,1]$.

\subsection{Noetherian schemes}
Similar to regularity, there is an \emph{absolute} notion of lci:
if $X$ is a noetherian scheme, then we say that $X$ is lci at $x$ if
the completion $\widehat{\sO}_{X,x}$ is a quotient of a regular local ring
$R$ by a regular sequence. If we write $\widehat{\sO}_{X,x}$
as a quotient $R/I$ of some regular ring $R$, which we always can do by Cohen's
structure theorem, then $I$ is generated by a regular sequence if and only if
$X$ is lci at $x$.

\subsection{Morphisms between noetherian schemes and stacks}\label{SS:lci:cotangent-complex}
Let
$f\colon X\to Y$ be a morphism of locally noetherian schemes. A \emph{Cohen
  factorization} of $f$ at $x\in |X|$ is a commutative diagram
\[
\xymatrix{
\Spec \widehat{\sO}_{X,x}\ar[d]\ar@{(->}[r]_-{i} & W\ar[d]^w\\
X\ar[r]_f & Y\\
}
\]
where $w$ is flat, the fiber $w^{-1}(f(x))$ is a regular scheme, and $i$ is a closed immersion. Cohen factorizations always
exist~\cite[Thm.~1.1]{avramov_Cohen}.

A morphism $f\colon X\to Y$ of locally noetherian schemes is lci at $x\in |X|$ if
$i$ is regular for some Cohen factorization. This does not depend on the Cohen
factorization and the following are equivalent~\cite[Thm.~1.2, (1.8)]{avramov_lci}
\begin{enumerate}
\item $f$ is lci at $x$,
\item $H^{-n}(\LL_f\Lotimes \kappa(x))=0$ for $n\geq 2$, and
\item $H^{-n}(\LL_f\Lotimes \kappa(x))=0$ for $n=2$.
\end{enumerate}

The latter characterizations show that the notion of lci is local in the smooth
topology on both $X$ and $Y$ and hence makes sense for morphisms of
locally noetherian algebraic stacks. To make sense of the conditions above
for stacks,
replace $H^{-n}(\LL_f\Lotimes \kappa(x))$ with $\coho^{-n}(\LDERF i_x^* \LL_f)$.

\subsection{Syntomic morphisms}
Let $f\colon \stX\to \stY$ be a morphism between algebraic stacks.  Assume that
either $\stX$ and $\stY$ are noetherian or that $f$ is locally of finite
presentation. Then we say that $f$ is \emph{syntomic} at $x\in |X|$ if it is
flat and lci at $x$. Equivalently, $f$ is flat at $x$ and the fiber
$f^{-1}(f(x))$ is lci at $x$. Indeed, in the noetherian case this follows from
the characterization using the cotangent complex since the cotangent complex
$\LL_f$ commutes with arbitrary base change since $f$ is flat. When $f$ is
locally of finite presentation, this is~\cite[Prop.~19.3.7]{egaIV} and then
the locus where $f$ is syntomic is open in $X$.

\subsection{Descent of lci morphisms}\label{S:descent-lci}

\newcommand{\stXhat}{\skew{5}\widehat{\stX}}
\newcommand{\Xhat}{\widehat{X}}

\begin{proposition}\label{P:main-thm:lci}
Let $f\colon \stX\to \stY$ be a morphism between stacks with
good moduli spaces and let $g\colon X\to Y$ denote the induced morphism
of good moduli spaces. Assume that either $f$ is locally of finite presentation
or that $\stX$ and $\stY$ are locally noetherian. Suppose that $f$ is strong.
If $f$ is lci, then so is $g$.
\end{proposition}
\begin{proof}
First assume that $f$ is locally of finite presentation. Then so is $g$
(Proposition~\ref{P:main-thm:properties}). The question is local on $X$
and $Y$ so we may assume that $X$ and $Y$ are affine and can factor
$g$ through a closed immersion $X\inj \AA^n_Y$. Then $g$ is lci if and only
if $X\inj \AA^n_Y$ is Koszul-regular and $f$ is lci if and only if
$\stX\inj \AA^n_\stY$ is Koszul-regular (see \S\ref{SS:lci:fp}).
The result now follows by Proposition~\ref{P:main-thm:regular-immersion}.

Instead assume that $\stX$ and $\stY$ are locally noetherian.  Let $x\in |X|$
be a point, let $y=g(x)$, let $\Xhat$ denote the completion at $x$ and let
$\Xhat\inj W\to Y$ be a Cohen factorization~\cite[Thm.~1.1]{avramov_Cohen}.
That is, $W\to Y$ is flat, the fiber $W_y$ is regular and $\Xhat\inj W$ is
a closed immersion.
Let $\stXhat\inj \stW\to \stY$ be the pull-back of the Cohen factorization
along $\pi_Y$. Note that $\stW_{y_0}=\stW\times_\stY \stG_{y_0}$ is regular.
Indeed, the morphism $\stG_{y_0}\to \Spec \kappa(y)$ is smooth since it is
an fppf gerbe.
If $f$ is lci, we conclude that $\stXhat\inj \stW$ is a
regular immersion and that $\Xhat\inj W$ is regular
(Proposition~\ref{P:main-thm:regular-immersion}) so that $g$ is lci.
\end{proof}

\end{section}


\bibliography{luna-fundamental}
\bibliographystyle{dary}

\end{document}